\documentclass[12pt]{amsart}
\usepackage{amsmath,amsthm,amsfonts,amssymb,mathrsfs}%\usepackage{paralist,nicefrac}
\usepackage[all]{xy}
\newcommand{\C}{\mathbb{C}}
\newcommand{\Z}{\mathbb{Z}}
\newcommand{\R}{\mathbb{R}}
\newcommand{\Q}{\mathbb{Q}}
\newcommand{\F}{\mathbb{F}}
\newcommand{\A}{\mathbb{A}}

\newcommand{\Gal}{\operatorname{Gal}}

\newcommand{\Ad}{\operatorname{Ad}}
\newcommand{\ad}{\operatorname{ad}}

\newcommand{\Spin}{\mathrm{Spin}}

\newcommand{\SO}{\mathrm{SO}}
\newcommand{\U}{\mathrm{U}}
\newcommand{\SU}{\mathrm{SU}}

\newcommand{\GL}{\mathrm{GL}}
\newcommand{\PGL}{\mathrm{PGL}}
\newcommand{\SL}{\mathrm{SL}}
\newcommand{\PSL}{\mathrm{PSL}}
\newcommand{\Sp}{\mathrm{Sp}}
\newcommand{\PSp}{\mathrm{PSp}}

\newtheorem{thm}{Theorem}[section]

\newtheorem{cor}[thm]{Corollary}

\newtheorem{prop}[thm]{Proposition}

\newtheorem{lem}[thm]{Lemma}
\theoremstyle{definition}
\newtheorem{definition}{Definition}
\begin{document}

\title{The Waring problem for Lie groups and Chevalley groups}

\author{Chun Yin Hui, Michael Larsen, Aner Shalev}

\address{Chun Yin Hui,
Einstein Institute of Mathematics,
Hebrew University, Givat Ram,
Jerusalem 91904, Israel}

\address{Michael Larsen,
Department of Mathematics,
Indiana University,
Bloomington, IN
47405,
U.S.A.}

\address{Aner Shalev,
Einstein Institute of Mathematics,
Hebrew University, Givat Ram,
Jerusalem 91904, Israel}

\thanks{The first author was supported by ERC Advanced Grant no. 247034.
The second author was partially supported by the Simons Foundation, the MSRI, 
the NSF, and BSF Grant no. 2008194. The third author was partially supported 
by ERC Advanced Grant no. 247034, ISF grant no. 1117/13, BSF Grant no. 2008194 
and the Vinik Chair of Mathematics which he holds}

\begin{abstract}
The classical Waring problem deals with expressing every natural 
number as a sum of $g(k)$ $k^{\mathrm {th}}$ powers. Similar problems 
were recently studied in group theory, where we aim to present group elements
as short products of values of a given word $w \ne 1$.
In this paper we study this problem for Lie groups and Chevalley groups
over infinite fields.

We show that for a fixed word $w \ne 1$ and for a classical connected 
real compact Lie group $G$ of sufficiently large rank we have $w(G)^2=G$, 
namely every element of $G$ is a product of $2$ values of $w$.

We prove a similar result for non-compact Lie groups of arbitrary rank,
arising from Chevalley groups over $\R$ 
or over a $p$-adic field. We also study this problem for 
Chevalley groups over arbitrary infinite fields, and show in particular
that every element in such a group is a product of two squares.
\end{abstract}

\maketitle

\newpage

\section{Introduction}

Let $F_d$ be the free group on $x_1, \ldots , x_d$ and let 
$w = w(x_1, \ldots , x_d) \in F_d$ be a word. For every
group $G$ there is a word map $w = w_G: G^d \rightarrow G$
obtained by substitution. The image of this map is denoted
by $w(G)$. The theory of word maps has developed
significantly in the past decade; see \cite{La,Sh1,S,LaSh1,LaSh2,LOST,LST,AGKSh,Sh2} and the references therein.

A major goal in these investigations is to prove theorems of ``Waring type'', i.e.,
to find small $k$ such that, 
for every word $w \ne 1$ and for various groups $G$ we have $w(G)^k = G$,
namely every element of $G$ is a product of $k$ values of $w$.

A theorem of Borel \cite{Bo1} 
states that if $w$ is a non-trivial word 
then the word map it induces on simple algebraic groups $G$ is
dominant. Thus $w(G)$ contains a dense open subset, which
easily implies $w(G(F))^2 = G(F)$ where $F$ is an algebraically
closed field (see Corollary 2.2 in \cite{Sh2}). However, much
more effort is required in order to prove similar results for
fields $F$ (finite or infinite) which are not algebraically closed.

In \cite{Sh1} it is shown that, fixing $w \ne 1$, we have
$w(G)^3 = G$ for all sufficiently large (nonabelian) finite simple 
groups $G$. This is improved in \cite{LaSh1, LaSh2, LST}
to $w(G)^2 = G$. Results of type $w(G)^3 = G$ were recently
obtained in \cite{AGKSh} for $p$-adic groups $G(\Z_p)$.

The purpose of this paper is to study similar problems
for Lie groups and for infinite Chevalley groups. 
Our main results are as follows.

\begin{thm} 
\label{T11}
For every two non-trivial words $w_1, w_2$ there
exists $N = N(w_1, w_2)$ such that if $G$ is a classical
connected real compact Lie group of rank at least $N$ then
\[
w_1(G)w_2(G) = G.
\]
\end{thm} 

In particular, for any $w \ne 1$ there is $N = N(w)$ such
that $w(G)^2 = G$ for all classical connected real compact
Lie groups of rank at least $N$.

We note that the assumption that the rank of $G$ is large
is necessary. 
By a theorem of E.~Lindenstrauss (private communication) and 
A.~Thom  \cite[Cor.~1.2]{T},
for any $n \ge 2$ and $\epsilon > 0$ there
exists a word $1 \ne w \in F_2$ such that all elements of 
$w(\U(n))$ have distance $\le \epsilon$ from the identity.  
Embedding a given $G$ in $\U(n)$, we can arrange 
that $w(G)^2 \neq G$ or even $w(G)^k \neq G$ for any fixed $k$.
We can also find a sequence $\{ w_i \}$ of
non-trivial words in two variables such that, for every compact group $G$,
$w_i(G)$ converges to $1$.

We also establish a width $2$  result  for non-compact Lie groups
which arise from Chevalley groups over $\R$ or over a $p$-adic field.
Here a Chevalley group over a field $F$ means a group generated by 
the root groups $X_\alpha(F)$ associated to a faithful representation of a complex semisimple Lie algebra (see \cite[$\mathsection3$]{St}),
or equivalently, the commutator subgroup of 
$G_F(F)$, where $G_F$
is a split semisimple algebraic group over $F$. In this case there is no large rank assumption.

\begin{thm} 
\label{T12}
Let $F$ be a field that contains either $\R$ or $\Q_p$ for 
some prime number $p$. Let $w_1$,$w_2$ be non-trivial words and $G$ 
a Chevalley  group over $F$ associated to a complex simple Lie algebra. Then
\[
G\setminus Z(G)\subset w_1(G)w_2(G).
\]
In particular, if $Z(G) = \{ 1 \}$, then $w_1(G)w_2(G) = G$.
\end{thm}

Without assumptions on the center of $G$ this result easily
implies $w_1(G)w_2(G)w_3(G) = G$ for any non-trivial words 
$w_1, w_2, w_3$.

Our last results deal with Chevalley groups over an arbitrary
infinite field $F$. Here we have a general width $4$ result, and width 3
and 2 in special cases.

\begin{thm} Let $w_1$,$w_2$,$w_3$,$w_4$ be non-trivial words
and let $F$ be an infinite field.
\begin{enumerate}
\item[(i)] If $G$ is
a Chevalley  group over $F$ associated to a complex simple Lie algebra, then \[
G\setminus Z(G)\subseteq w_1(G)w_2(G)w_3(G)w_4(G).
\]
In particular, if $Z(G) = \{ 1 \}$, then $w_1(G)w_2(G)w_3(G)w_4(G) = G$. 
\item[(ii)] If $G=\SL_n(F)$ and $n>2$, then
\[
G\setminus Z(G)\subseteq w_1(G)w_2(G)w_3(G).
\]
Hence,  $w_1(\PSL(n,F))w_2(\PSL(n,F))w_3(\PSL(n,F)) = \PSL(n,F)$. 
\end{enumerate}
\end{thm}

For some specific words we obtain stronger results.

\begin{thm} Let $w_1 = x^m$ and $w_2 = y^n$ where $m, n$ are positive 
integers. Let $G$ be a Chevalley group over an infinite field $F$ associated to a complex semisimple Lie algebra $\mathfrak{g}$. 
\begin{enumerate}
\item[(i)] If $\mathfrak{g}$ is simple, then
\[
G \setminus Z(G)\subseteq w_1(G)w_2(G).
\]
In particular, if $Z(G) = 1$ then $w_1(G)w_2(G) = G$.
\item[(ii)] If $m=n=2$, then
\[
G = w_1(G)w_2(G).
\]
\end{enumerate}
\end{thm}

We also give an example showing that a non-trivial central element 
is not in the image of the word map $x^4y^4$ (Proposition 4.1). 
See also \cite{LaSh3} for the probabilistic behavior of word maps 
induced by $x^my^n$ on finite simple groups.

The fact that every element of $G$ above is a product of two squares 
can be regarded as a non-commutative analogue of Lagrange's four squares 
theorem. A similar result for finite quasisimple groups can be found
in \cite{LST2}.

This paper is organized as follows. In Section 2 we deal with
compact Lie groups and prove Theorem 1.1. Section 3 is devoted
to the proof of Theorem 1.2, and Theorems 1.3 and 1.4 are
proved in Section 4.

\section{Compact Lie groups}

In this section we provide solutions for Waring type problems with length 
two for classical connected real compact Lie groups $G$ with large rank,
thus proving Theorem~\ref{T11}. It suffices to work with simply connected groups $G$, i.e.
with $\SU(n)$, $\Sp(n)$, and $\Spin(n)$. Let us start with Got\^o's theorem.

\begin{thm}\cite{Go} 
Let $G$ be a connected compact semisimple Lie group and $T$ a maximal torus of $G$. 
Then there exists $x\in N_G(T)$ such that $\Ad(x)-1$ is non-singular on $\mathrm{Lie}(T)$. 
Hence, every element $g$ of $G$ is conjugate to $[x,t]:=xtx^{-1}t^{-1}$ for some $t\in T$.  
\end{thm}

Let $w_1$ and $w_2$ be non-trivial words. Every element of $G$ can be conjugated into $T$ so the width two result for $G$
follows if we can prove $T\subset w_1(G) w_2(G)$.
By Got\^o's theorem, it suffices to show
 that $x\in w_1(G)$ and $x^{-1}\in w_2(G)$. This will be achieved by using the principal 
homomorphism \cite{Se}. Identify $S^1$, the subgroup of unit circle of $\C^*$ as a maximal torus of $\SU(2)$.

\begin{lem}
The primitive $n$th roots of unity 
$\zeta_n^{\pm 1}:=e^{\pm \frac{2\pi i}{n}}$ both belong to 
$w_i(\SU(2))\cap S^1$ for $i=1,2$ if $n$ is sufficiently large.
\end{lem}

\begin{proof}
Since $w_1$ and $w_2$ are non-trivial, $w_i(\SU(2))$ contains a non-empty open subset of $\SU(2)$ for $i=1,2$ \cite[Cor. 5]{La}. As 
$\SU(2)$ is compact and connected and $x$ and $x^{-1}$ are conjugate for any $x\in\SU(2)$, it follows that $w_i(\SU(2))\cap S^1$ is
a closed arc and also a symmetric neighborhood of $1$ in $S^1$ for $i=1,2$. Hence, the primitive $n$-th roots of unity $\zeta_n,\zeta_n^{-1}\in S^1$ belong to $w_i(\SU(2))$ for $i=1,2$ if $n$ is sufficiently large.
\end{proof}

\begin{definition}
We make the following definitions.
\begin{enumerate}
\item Let $I_n$ be the identity complex $n\times n$ matrix.
\item Let $0_n$ be the zero complex $n\times n$ matrix.
\item Let $E_n^i$ be the diagonal complex $n\times n$ matrix whose $(i,i)$-entry is $1$ and all other entries $0$.
\item Let $L_n^i$ be the linear functional on diagonal complex $n\times n$ matrices  such that $L_n^i(E_n^j)=\delta_{ij}$ (Kronecker delta) for all $1\leq j\leq n$.
\item Let $s_n\in\U(n)$ be the $n$-cycle
$$\begin{pmatrix}
      0 &0 &...&0 & 1\\
      1 &0&... &0  &0\\
      0 &1&...&0& 0\\
      \vdots&\vdots &\ddots &\vdots &\vdots\\
      0 &0 &... &1 &0
\end{pmatrix}.$$
\end{enumerate}
\end{definition}

\begin{thm}
\label{T25} For any non-trivial words $w_1,w_2$ and a sufficiently large $n$
we have
\begin{equation*}
\SU(n)=w_1(\SU(n))w_2(\SU(n)).
\end{equation*}\end{thm}

\begin{proof} 
Consider the commutative diagram
\begin{equation*}
\xymatrix{
\SL(2,\C) \ar[rd]^{p} \ar[r]^{\widetilde{p}}
&\SL(n,\C) \ar@{->>}[d]^{\pi}\\
&\PSL(n,\C) }
\end{equation*}
where $p$ is the principal homomorphism associated to simple roots \cite{FH}
$$\Delta:=\{L^{1}_n-L_n^{2},L_n^2-L_n^3,...,L_n^{n-1}-L_n^{n}\},$$
 $\pi$ the adjoint quotient, and $\widetilde{p}$ a lifting of $p$. Let $H=\begin{pmatrix}
      1  & 0\\
      0  &-1
\end{pmatrix}\in\mathfrak{sl}(2,\C)$. Homomorphism $\widetilde{p}$ is isomorphic to $\mathrm{Sym}^{n-1}$ since $\alpha(d\widetilde{p}(H))=2$ for all $\alpha\in\Delta$ \cite[$\mathsection2.3$]{Se}. By restricting to suitable maximal 
compact subgroups, we obtain $\widetilde{p}:\SU(2)\rightarrow \SU(n)$. Identify $S^1$ as a maximal torus of $\SU(2)$. The set of eigenvalues of  $\widetilde{p}(\zeta_m)$ is 
\begin{equation*}
\{\zeta_m^{n-1},\zeta_m^{n-3},...,\zeta_m^{3-n},\zeta_m^{1-n}\}.
\end{equation*}
Then $\widetilde{p}(\zeta_{2n})$ is conjugate to $x_n:=s_n$ 
if $n$ is odd and $\widetilde{p}(\zeta_{2n})$ is conjugate to 
$x_n:=s_n\cdot\zeta_{2n}I_{n}$
if $n$ is even by comparing eigenvalues. Let $T$ be the diagonal maximal torus of $\SU(n)$. Since $\Ad(x_n)-1$ is non-singular on $\mathrm{Lie}(T)$ for all $n$, $x_n\in w_1(\widetilde{p}(\SU(2)))$ and $x_n^{-1}\in w_2(\widetilde{p}(\SU(2)))$ for $n\gg0$  by Lemma 2.2, we are done by Got\^o's theorem.
\end{proof}

Next, we work with the real compact symplectic group 
$$\Sp(n):=\U(2n)\cap \Sp(2n,\C),$$
where $\Sp(2n,\C)$ is the subgroup of $\GL(2n,\C)$ that preserves the form 
$$\begin{pmatrix}
      0_n  &I_n\\
      -I_n  &0_n
\end{pmatrix}.$$

\begin{thm}
\label{T27}
For any non-trivial words $w_1,w_2$ and a sufficiently large $n$ 
we have
\begin{equation*}
\Sp(n)=w_1(\Sp(n))w_2(\Sp(n)).
\end{equation*}
\end{thm}

\begin{proof} Let $T$ be the maximal torus of $\Sp(n)$ consisting of diagonal matrices with complex entries. 
Let $x_n\in N_{\Sp(n)}(T)$ be the element
$$\begin{pmatrix}
      s_n  &0_n\\
      0_n  &s_n
\end{pmatrix}\cdot
\begin{pmatrix}
      I_n-E_n^1  &E_n^1\\
      -E_n^1  & I_n-E_n^1
\end{pmatrix}.$$
Then it is easy to see that $\Ad(x_n)-1$ is non-singular on $\mathrm{Lie}(T)$. By Got\^o's Theorem, it suffices to show
that $x_n\in w_1(\Sp(n))$ and $x_n^{-1}\in w_2(\Sp(n))$ for all sufficiently large $n$.

Consider the commutative diagram
\begin{equation*}
\xymatrix{
\SL(2,\C) \ar[rd]^{p} \ar[r]^{\widetilde{p}}
&\Sp(2n,\C) \ar@{->>}[d]^{\pi}\\
&\PSp(2n,\C) }
\end{equation*}
where $p$ is the principal homomorphism associated to simple roots  \cite{FH} $$\Delta:=\{L^{1}_{2n}-L_{2n}^{2},L_{2n}^2-L_{2n}^3,...,L_{2n}^{n-1}-L_{2n}^{n},2L_{2n}^n\},$$ $\pi$ the adjoint quotient, and $\widetilde{p}$ a lifting of $p$. Let $H=\begin{pmatrix}
      1  & 0\\
      0  &-1
\end{pmatrix}\in\mathfrak{sl}(2,\C)$. Since $\alpha(d\widetilde{p}(H))=2$ for all $\alpha\in\Delta$ \cite[$\mathsection2.3$]{Se}, the 
set of weights of $\widetilde{p}$, viewed as a $2n$-dimensional representation, is
$$\{2n-1,2n-3,...,1,-1,...,3-2n,1-2n\}.$$
By restricting to suitable maximal 
compact subgroups, we obtain $\widetilde{p}:\SU(2)\rightarrow \Sp(n)$. 

Identify $S^1$ as a maximal torus of $\SU(2)$. The set of eigenvalues of  $\widetilde{p}(\zeta_m)$ is 
\begin{equation*}
\{\zeta_m^{2n-1},\zeta_m^{2n-3},...,\zeta_m^{3-2n},\zeta_m^{1-2n}\}.
\end{equation*}
Define $e_1:=(1,0,...,0)\in \C^{2n}$. Since $x_n$ satisfies $x_n^{2n}+1=0$ and the set of vectors
$$\{x_ne_1,x_n^2e_1,...,x_n^{2n}e_1\}$$
is linearly independent, the characteristic polynomial of $x_n$ is $t^{2n}+1$. 
In $\Sp(n)$, the characteristic polynomial determines the conjugacy class.  (Indeed, the diagonal matrices with entries 
$\lambda_1,\ldots,\lambda_n, \lambda_1^{-1},\ldots,\lambda_n^{-1}$ form a single orbit under the action of the Weyl group.)
Since $\widetilde{p}(\zeta_{4n})$ and $x_n$ have the same characteristic polynomial, it follows 
that they are conjugate in $\Sp(n)$. Hence, $x_n$ and $x_n^{-1}$ respectively belong to $w_1(\widetilde{p}(\SU(2)))$ and $w_2(\widetilde{p}(\SU(2)))$ when $n$ is sufficiently large by Lemma 2.2. We are done.
\end{proof}

We then consider compact special orthogonal group $\SO(n)$ and its simply connected cover $\Spin(n)$ for $n\geq3$.

\begin{thm}For any non-trivial words $w_1,w_2$ and a sufficiently large
$n$ we have
\begin{equation*}
\SO(n)=w_1(\SO(n))w_2(\SO(n)).
\end{equation*}
\end{thm}

\begin{proof} 
Since we have a morphism $\SO(2n)\rightarrow \SO(2n+1)$ such that the image of a maximal torus of $\SO(2n)$ is a maximal torus of $\SO(2n+1)$, it suffices to deal with $\SO(2n)$. This is a maximal compact subgroup of $\SO(2n,\C)$. Let $K(2n,\C)$ be the subgroup of $\SL(2n,\C)$ preserving
the form
$$\begin{pmatrix}
      0_n  &I_n\\
      I_n  &0_n
\end{pmatrix}.$$ 
Since $K(2n,\C)$ is isomorphic to $\SO(2n,\C)$ \cite{FH} and has a diagonal maximal torus, we use $K(2n,\C)$ and $K(2n):=\U(2n)\cap K(2n,\C)$ (a maximal compact of $K(2n,\C)$) instead of $\SO(2n,\C)$ and $\SO(2n)$. One checks that the diagonal maximal torus $T$ of $K(2n)$ is equal to the diagonal maximal torus of $\Sp(n)$. Let $s_{n-1}'$ be the $n\times n$ matrix 
$$\begin{pmatrix}
      1 &0 &...&0 & 0\\
      0 &0&... &0  &1\\
      0 &1&...&0& 0\\
      \vdots&\vdots &\ddots &\vdots &\vdots\\
      0 &0 &... &1 &0
\end{pmatrix}$$
which fixes $e_1:=(1,0,...,0)\in\C^n$ and is an $(n-1)$-cycle on the natural complement of $e_1$ in $\C^n$. Let $x_n\in N_{K(2n)}(T)$ be the element
$$\begin{pmatrix}
      s_{n-1}'  &0_n\\
      0_n  &s_{n-1}'
\end{pmatrix}\cdot
\begin{pmatrix}
      I_n-E_n^1-E_n^2  & E_n^1+E_n^2\\
      E_n^1+E_n^2  & I_n-E_n^1-E_n^2
\end{pmatrix}.$$
By choosing the basis 
$$\{E_{2n}^1-E_{2n}^{n+1},E_{2n}^2-E_{2n}^{n+2},...,E_{2n}^n-E_{2n}^{2n}\}$$
for $\mathrm{Lie}(T)$, the $\Ad(x_n)$-action on $\mathrm{Lie}(T)$ is given by the $n\times n$ matrix
$$\begin{pmatrix}
      -1 &0 & 0 &...&0 & 0\\
      0 &0& 0&... &0  &1\\
      0 &-1&0&...&0& 0\\
      0 &0&1&...&0& 0\\
      \vdots&\vdots&\vdots &\ddots &\vdots &\vdots\\
      0 &0&0 &... &1 &0
\end{pmatrix}.$$
One sees that $\Ad(x_n)-1$ is non-singular on $\mathrm{Lie}(T)$. By Got\^o's Theorem, it suffices to show
that $x_n\in w_1(K(2n))$ and $x_n^{-1}\in w_2(K(2n))$ for all sufficiently large $n$. Since $x_n$
is conjugate in $\GL(2n,\C)$ to the permutation matrix
$$\begin{pmatrix}
      0 &1 & 0 &...&0 & 0\\
      1 &0& 0&... &0  &0\\
      0 &0&0&...&0& 1\\
      0 &0&1&...&0& 0\\
      \vdots&\vdots&\vdots &\ddots &\vdots &\vdots\\
      0 &0&0 &... &1 &0
\end{pmatrix},$$
the characteristic polynomial of $x_n$ is $(t^2-1)(t^{2n-2}-1)$. 

Consider the commutative diagram
\begin{equation*}
\xymatrix{
\SL(2,\C) \ar[rd]^{p} \ar[r]^{\widetilde{p}}
&K(2n,\C) \ar@{->>}[d]^{\pi}\\
&K(2n,\C)/\{\pm I_{2n}\} }
\end{equation*}
where $p$ is the principal homomorphism associated to simple roots  \cite{FH} $$\Delta:=\{L^{1}_{2n}-L_{2n}^{2},L_{2n}^2-L_{2n}^3,...,L_{2n}^{n-1}-L_{2n}^{n},L_{2n}^{n-1}+L_{2n}^{n}\},$$ $\pi$ the adjoint quotient, and $\widetilde{p}$ a lifting of $p$. Let $H=\begin{pmatrix}
      1  & 0\\
      0  &-1
\end{pmatrix}\in\mathfrak{sl}(2,\C)$. Since $\alpha(d\widetilde{p}(H))=2$ for all $\alpha\in\Delta$ \cite[$\mathsection2.3$]{Se}, the multiset of weights of the $2n$-dimensional representation $\widetilde{p}$ is 
$$\{2n-2,2n-4,...,2,0,0,-2,...,4-2n,2-2n\}.$$
By restricting to suitable maximal 
compact subgroups, we obtain $\widetilde{p}:\SU(2)\rightarrow K(2n)$. Identify $S^1$ as a maximal torus of $\SU(2)$. The set of eigenvalues of  $\widetilde{p}(\zeta_m)$ is 
\begin{equation*}
\{\zeta_m^{2n-2},\zeta_m^{2n-4},...,\zeta_m^{2},1,1,\zeta_m^{-2},...,\zeta_m^{4-2n},\zeta_m^{2-2n}\}.
\end{equation*}

It is not in general true that two diagonal orthogonal matrices are conjugate in $\SO(2n)$ if and only if they have the same characteristic
polynomial, because the Weyl group of $\SO(2n)$ is only an index $2$ subgroup of $S_n\ltimes (\Z/2\Z)^n$.
However, it is true in the case of matrices for which $1$ is an eigenvalue.
Since $\widetilde{p}(\zeta_{4n-4})$ and $x_n$ both have  characteristic polynomial 
$$(t^2-1)(t^{2n-2}-1)$$ 
and have eigenvalue $1$, it follows that they are conjugate in $K(2n)$. Hence, $x_n$ and $x_n^{-1}$ respectively belong to $w_1(\widetilde{p}(\SU(2)))$ and $w_2(\widetilde{p}(\SU(2)))$ when $n$ is sufficiently large by Lemma 2.2. We are done.
\end{proof}

\begin{thm}For any non-trivial words $w_1,w_2$ and a sufficiently large
$n$ we have
\begin{equation*}
\Spin(n)=w_1(\Spin(n))w_2(\Spin(n)).
\end{equation*}
\end{thm}

\begin{proof}
Since we have a morphism $\Spin(2n)\rightarrow \Spin(2n+1)$ such that the image of a maximal torus of $\Spin(2n)$ is a maximal torus of $\Spin(2n+1)$, it suffices to deal with $\Spin(2n)$. Consider commutative diagram 
$$\xymatrix{
\SU(2) \ar[rd]^{\widetilde{p}} \ar[r]^{\hat{p}}
&\Spin(2n) \ar@{->>}[d]^{\widetilde{\pi}}\\
&K(2n) }$$ 
where $\widetilde{\pi}$ is the natural projection and $\hat{p}$ is a lift of $\widetilde{p}$. Recall maximal torus $T$ and element $x_n\in N_{K(2n)}(T)$ constructed in the proof of Theorem 2.5. Since $x_n$ and $x_n^{-1}$ respectively belong to $w_1(\widetilde{p}(\SU(2)))$ and $w_2(\widetilde{p}(\SU(2)))$ when $n$ is sufficiently large, $w_1(\Spin(2n))$ and $w_2(\Spin(2n))$ respectively contain $\hat{x}_n$ and $\hat{x}_n^{-1}$ such that $\widetilde{\pi}(\hat{x}_n)=x_n$ by the diagram. Let $\hat{T}$ be the maximal torus of $\Spin(2n)$ such that $\widetilde{\pi}(\hat{T})=T$. We also have $\hat{x}_n\in N_{\Spin(2n)}(\hat{T})$. Consider commutative diagram

$$\xymatrixcolsep{5pc}\xymatrix{
\mathrm{Lie}(\hat{T}) \ar[d]^{d\widetilde{\pi}} \ar[r]^{\Ad(\hat{x}_n)-1} &\mathrm{Lie}(\hat{T})\ar[d]^{d\widetilde{\pi}}\\
\mathrm{Lie}(T) \ar[r]^{\Ad(x_n)-1} & \mathrm{Lie}(T)}$$
Since $d\widetilde{\pi}$ and $\Ad(x_n)-1$ are non-singular, $\Ad(\hat{x}_n)-1$ is also non-singular. We are done by Got\^o's Theorem.
\end{proof}

We end this section with a width two result for any compact semisimple real Lie group $G$.

\begin{thm}
Let $G$ be a compact semisimple real Lie group and $w_1,w_2$ non-trivial words. 
\begin{enumerate}
\item[(i)] If $\zeta_4\in S^1\cap w_1(\SU(2)) \cap w_2(\SU(2)$, then $w_1(G)w_2(G)=G$.
\item[(ii)] If $w_1^2(\SU(2))=\SU(2)$, then $w_1^2(G)=G$.
\end{enumerate}
\end{thm}

\begin{proof}
(i) Since $S^1\cap w_i(\SU(2))$ ($i=1,2$) is a connected closed arc, symmetric about the $x$-axis of the complex plane (Lemma 2.2), we obtain
$$\zeta_{2n}^{\pm}\in S^1\cap w_1(\SU(2)) \cap w_2(\SU(2)$$
 for all $n\geq 2$ by the assumption. Hence, $w_1(\SU(n))w_2(\SU(n))=\SU(n)$ for all $n\geq 2$ by the proof of Theorem 2.3. Since every element of $G$ is conjugate to some element in a maximal torus and $G$ contains an equal rank semisimple subgroup $H$ with type A simple factors, we are done.
 
(ii) There exists $x_1,x_2\in w_1(\SU(2))$ such that $x_1x_2=-1$. We may assume $x_1,x_2\in S^1$. Then one sees easily that $\zeta_4\in S^1\cap w_1(\SU(2))$ by Lemma 2.2. We obtain $w_1^2(G)=G$ by (i)
\end{proof}

\section{Non-compact groups}

In this section we study Waring type problems for split semisimple Lie groups 
$G$ over a local field of characteristic $0$ ($\R$, $\C$, or finite extension of $\Q_p$) and prove Theorem~\ref{T12}. 
A key result we need (related to the Thompson 
conjecture) was proved by Ellers and Gordeev \cite{EG1,EG2,EG3}:

\begin{thm} 
\label{T31}
Let $G$ be a Chevalley group over a 
field $F$ associated to complex simple Lie algebra. Let $g_1$ and $g_2$ be two regular semisimple elements in $G$ 
from a maximal split torus and let $C_1$ and $C_2$ be the conjugacy classes 
of $g_1$ and $g_2$, respectively. Then
\begin{equation*}
 G\setminus Z(G)\subseteq C_1C_2.
\end{equation*}
\end{thm}

\vspace{.2in}

In order to prove Theorem~\ref{T12} we also need the following.

\begin{lem} 
\label{L32}
Let $F$ be a field that contains either $\R$ or $\Q_p$ for some $p$ and $w$ a non-trivial word of $d$ letters. Denote by $H$ the group $\SL_2(F)$ and by $\mathrm{Tr}$. the trace map on $H$.  Define the discriminant on $H^d$ as
\begin{equation*}
\Delta:=(\mathrm{Tr}\circ w)^2-4.
\end{equation*}
Then $\Delta(H^d)\cap (F^\times)^2$ is an infinite set.
\end{lem}

\begin{proof} 
If $A\in \SL_2(\Q)$ and $N\in M_2(\Q)$ has trace $0$, then 
$$f_{A,N}\colon t\mapsto (I_2+tN)A$$ 
defines a non-constant morphism $\A^1\to \SL_2$, and the tangent spaces
to the curve $f_{A,N}(\A^1)$ at $A$, as $N$ ranges over trace $0$ matrices, span the tangent space to $\SL_2$ at $A$.
Thus, for every $(A_i)\in \SL_2(\Q)^d$ and every proper subvariety $V$ of $\SL_2^d$ which contains $(A_i)$ as a non-singular point,
there exists a non-constant morphism $f\colon \A^1\to \SL_2^d$ defined over $\Q$ such that $f(0) = (A_i)$,
and $f(\A^1)$ is not contained in $V$.
Any non-constant morphism $\pi\colon \SL_2^d\to \A^1$ defined over $\Q$ is generically smooth, and $\SL_2(\Q)^d$ is Zariski-dense in $\SL_2$,
so there always exists $(A_i)\in \SL_2(\Q)^d$ which is a non-singular point of the fiber of $\pi$ to which is belongs.  Thus, there exists
$f\colon \A^1\to \SL_2^d$ defined over $\Q$ so that the composition $\pi\circ f$ is a non-constant morphism $\A^1\to \A^1$.

We apply this general observation to the case $\pi := \mathrm{Tr}\circ w$ to obtain a morphism $f$ and a polynomial $P(x)\in \Q[x]$, not necessarily zero, such that
$\mathrm{Tr}\circ w(f(x)) \equiv P(x)$.  To prove the lemma, it suffices to show that $y^2 = P(x)^2-4$ has infinitely many solutions in the field $F$.
We write $P(x) = c_0 + \cdots + c_k x^k$, where $c_k\neq 0$ and $k\ge 1$.
Consider the curve $X$ over $\Q$ given in projective coordinates by 
$$c_k^2 u^{2k} - \Bigl(\sum_{i=0}^k c_i v^i w^{k-i}\Bigr)^2 + 4w^{2k}.$$
As $P := (1:1:0)$ is a non-singular point, by the (real or $p$-adic) implicit function theorem, there is an infinite (real or $p$-adic) neighborhood of $P$
in $X(F)$.   Letting $y := c_k u^k/w^k$ and $x := v/w$, this implies 
that $y^2 = P(x)^2-4$ has infinitely many solutions in $F$.

\end{proof}

\noindent \textbf{\textit{Proof of Theorem~\ref{T12}}:}

\begin{proof} In light of Theorem~\ref{T31}, it suffices to prove the theorem for $F=\R$ or $\Q_p$. Suppose $w$ is a non-trivial word of $d$ letters. Let $D$ be the group of diagonal matrices in $\SL_2(F)$. 
Any Chevalley group $G$ over $F$ is the commutator subgroup of the group of $F$-rational points of a corresponding simple algebraic group $G_F$, and we have the following commutative diagram of algebraic groups over $F$:

\begin{equation}
\xymatrix{
\SL_{2,F} \ar[d]^{\pi_1} \ar[r]^{\tilde{p}} &G_{F}\ar[d]^{\pi_2}\\
\PGL_{2,F} \ar[r]^p & G_{F}^{\ad}}
\end{equation}
where $\pi_1$ and $\pi_2$ are adjoint quotient maps, $G_{F}^{\ad}$ is the adjoint group of $G_{F}$, $p$ is the principal homomorphism associated to a system of simple roots \cite[$\mathsection2$]{Se}, and $\tilde{p}$ is a lifting of $p$. Since $w(\SL_2(F))$ contains infinitely many elements in $D$ by Lemma~\ref{L32} and the image of a generic element of $\pi_1(D)\subset\PGL_2(F)$ under $p$ is regular \cite[$\mathsection2.3$]{Se}, $w(\tilde{p}(\SL_2(F)))$ contains  a regular split semisimple element.
This semisimple element belongs to $w(G)$ since $\SL_2(F)$ is equal to its commutator subgroup \cite{Th}.
Therefore, we obtain
\begin{equation*}
G\setminus Z(G)\subseteq w_1(G)w_2(G)\end{equation*}
for non-trivial words $w_1$ and $w_2$ by Theorem~\ref{T31}. 
\end{proof}

We now state some easy consequences of Theorem~\ref{T12}.

\begin{cor}
\label{C33}Let $w_1$,$w_2$,$w_3$ be non-trivial words and $G$ a Chevalley group over a local field of characteristic $0$ associated to complex semisimple Lie algebra $\mathfrak{g}$. Then
\begin{equation*}
G=w_1(G)w_2(G)w_3(G).
\end{equation*}
\end{cor}

\begin{proof} Let $\widetilde{G}$ be the universal group \cite[p.45]{St} associated to $\mathfrak{g}$. We have central extension $\pi:\widetilde{G}\rightarrow G$ \cite[$\mathsection7$]{St} and $\widetilde{G}$ is the direct product of universal groups $\widetilde{G_i}$ associated to simple factors of $\mathfrak{g}$. It suffices to prove the corollary for $\widetilde{G_i}$.   Since the center $Z(\widetilde{G_i})$ is finite and $w_3(\widetilde{G_i})$ has non-empty interior (same proof as \cite[Cor. 5]{La}), Corollary~\ref{C33} follows easily from Theorem~\ref{T12}.\end{proof}

\begin{cor}Let $w_1$,$w_2$ be non-trivial words and $G^{\ad}:=G/Z(G)$ where $G$ is a Chevalley group over $F$ associated to complex semisimple Lie algebra $\mathfrak{g}$. Then
\begin{equation*}
G^{\ad}=w_1(G^{\ad})w_2(G^{\ad}).
\end{equation*}
\end{cor}

\begin{proof} Let $\widetilde{G}$ be the universal group associated to $\mathfrak{g}$. We have central extension $\pi:\widetilde{G}\rightarrow G$ and $\widetilde{G}$ is the direct product of universal groups $\widetilde{G_i}$ associated to simple factors of $\mathfrak{g}$. Since $\widetilde{G}/Z(\widetilde{G})=G/Z(G)$, it suffices to prove the corollary for $\widetilde{G_i}$ which is done by Theorem~\ref{T12}.
\end{proof}

\section{Chevalley groups}

The method we used in $\mathsection3$ leads to the proof
of Theorem 1.3, given below.

\begin{proof} (i) Applying Theorem 3.1 it suffices to show that 
$w_1(G)w_2(G)$ and $w_3(G)w_4(G)$ contain split regular semisimple elements. 
By the principal homomorphism and diagram (5), it suffices to show that 
\begin{equation*}
w_1(\SL_2(F))w_2(\SL_2(F))
\end{equation*}
contains an infinite set of split semisimple elements of $\SL_2(F)$ of different 
traces. Since $F$ is infinite and words are non-trivial, $w_1(\SL_2(F))$ and $w_2(\SL_2(F))$ both contain infinitely many semisimple elements of different traces. If either  $w_1(\SL_2(F))$ or $w_2(\SL_2(F))$ contains infinitely many split semisimple elements of different traces, then we are done. Otherwise, let $C_1$ and $C_2$ be conjugacy classes respectively of non-split semisimple elements of $w_1(\SL_2(F))$ and $w_2(\SL_2(F))$. Then the diagonal matrix $\mathrm{diag}(\lambda,\lambda^{-1})\in C_1C_2$ if and only if $-\lambda\in\chi(C_1)\chi(C_2)$ \cite[Lemma 6.2]{VW}, where $\chi(C_i)$ is the set of $(2,1)$ entries of $C_i$ (\emph{corner invariant}) for $i=1,2$ \cite[$\mathsection3$]{VW}. Since $F$ is infinite, the corner invariants $\chi(C_1),\chi(C_2)$ are infinite and we are done.\\

(ii) An $n$ by $n$ matrix $M$ is said to be \emph{cyclic} if every Jordan block of $M$ is of multiplicity one. Let $G$ be $\SL_n(F)$ with $n>2$. A conjugacy class $C$ of $G$ is cyclic if every element of $C$ is cyclic. If $C_1,C_2,C_3$ are cyclic conjugacy classes of $G$, then any non-scalar element of $G$ belongs to the product $C_1C_2C_3$ \cite[Theorem 3]{Lev}. Therefore, it suffices to show that $w_1(G)$ contains a cyclic element. Since $F$ is infinite and $w_1$ is non-trivial, $w_1(G)$ contains a regular semisimple (and therefore cyclic) element $g_1$ by the principal homomorphism and diagram (5).  \end{proof}
\medskip

Let us now prove Theorem 1.4. 

\begin{proof}
(i) Since $G$ is a Chevalley group over
an infinite field, it has a maximal split torus $T$. Recall that
$w_1=x^{m}$ and $w_2=y^{n}.$ 
It is easy to see that $w_1(T)$ and $w_2(T)$ both contain split regular 
semisimple elements. Thus, for $i=1,2$, $w_i(G)$ contains a conjugacy 
class of $C_i$ of a split regular semisimple element. 
By Theorem 3.1, we obtain 
$G \setminus Z(G) \subseteq w_1(G)w_2(G)$, proving the
result.\\

(ii) Let $\widetilde{G}$ be the universal group associated to $\mathfrak{g}$. We have central extension $\pi:\widetilde{G}\rightarrow G$  and $\widetilde{G}$ is the direct product of universal groups $\widetilde{G_i}$ associated to simple factors of $\mathfrak{g}$. We just need to deal with the case that $G$ is universal and $\mathfrak{g}$ is simple. Since $G \setminus Z(G)\subset w_1(G)w_2(G)$ by (i), it suffices to show $Z(G)\subset w_1(G)w_2(G)$.  Let $\Lambda$ and $R$ be respectively the weight lattice and root lattice of $\mathfrak{g}$. We have \cite[p.45]{St}
\begin{equation*}
Z(G)=\mathrm{Hom}(\Lambda/R, F^*)
\end{equation*}
and
\begin{center}
\begin{tabular}{c|c} 
Type of $\mathfrak{g}$  & $\Lambda/R$ \\ \hline

$A_n=\mathfrak{sl}(n+1)$ ($n\geq 1$) & $\Z/(n+1)\Z$ \\ 

$B_n=\mathfrak{so}(2n+1)$ ($n\geq 2$) & $\Z/2\Z$ \\

$C_n=\mathfrak{sp}(2n)$ ($n\geq 3$) & $\Z/2\Z$ \\ 

$D_n=\mathfrak{so}(2n)$ ($n\geq 5$, odd) & $\Z/4\Z$ \\ 

$D_n=\mathfrak{so}(2n)$ ($n\geq 4$, even) & $\Z/2\Z\times\Z/2\Z$ \\ 

$E_6$  &  $\Z/3\Z$ \\ 

$E_7$  & $\Z/2\Z$ \\ 

$E_8$  & $\{1\}$ \\ 

$F_4$  & $\{1\}$ \\ 

$G_2$  & $\{1\}$ \\ 
\end{tabular}
\end{center}

Since $|\Lambda/R|$ is odd if $\mathfrak{g}=\mathfrak{sl}(2n+1),E_6,E_8,F_4,G_2$ ($n\geq1)$, every element of $Z(G)$ is a square in these cases. If $\mathrm{char}(F)=2$, then $Z(G)$ is trivial for the remaining $\mathfrak{g}$ and (ii) is true. Assume $\mathrm{char}(F)\neq 2$.\\

Case $\mathfrak{g}=\mathfrak{sl}(2n)$ ($n\geq1$):\\ 
We have $G=\SL_{2n}(F)$ and 
$Z(G)=\{rI_{2n}:~ r\in F,~r^{2n}=1\}$.
Define $J_r:=\begin{pmatrix}
      0 &1\\
      r &0 
\end{pmatrix}$ whenever $r\in F$ is a $2n$-th (not necessary primitive) root of unity.
When $n=1$, the non-trivial center of $G$ is a square since
\begin{equation*}
   \begin{pmatrix}
      -1 &0\\
      0 & -1
\end{pmatrix}=J_{-1}^2.
\end{equation*}
When $n>1$, every center element of $\SL_{2n}(F)$ is a product of $2$ squares since
\begin{equation*}
   rI_{2n}=
\begin{pmatrix}
      J_r &... &0  &0\\
      \vdots &\ddots &\vdots &\vdots\\
      0  &... &J_r &0\\
      0  &...&0 &J_{r}
\end{pmatrix}^2
\end{equation*}
if $(-r)^n=1$ and 
\begin{equation*}
   rI_{2n}=\begin{pmatrix}
      I_2  &... &0 &0\\
      \vdots &\ddots&\vdots  &\vdots\\
      0 &...  &I_2   &0\\
      0 &... &0 & J_{-1}
\end{pmatrix}^2
\begin{pmatrix}
      J_r &... &0  &0\\
      \vdots &\ddots &\vdots &\vdots\\
      0  &... &J_r &0\\
      0  &...&0 &J_{-r}
\end{pmatrix}^2
\end{equation*}
if $(-r)^n=-1$.\\

Case $\mathfrak{g}=\mathfrak{so}(2n+1)$ ($n\geq2),\mathfrak{sp}(2n)$ ($n\geq3),\mathfrak{so}(4n)$ ($n\geq2),E_7$:\\
By using the fact that 
\begin{itemize}
\item $\mathfrak{sp}(2)=\mathfrak{sl}(2)=\mathfrak{so}(3)$,
\item $\mathfrak{sl}(2)\times \mathfrak{sl}(2)=\mathfrak{so}(4)\subset\mathfrak{so}(5)=\mathfrak{sp}(4)$, 
\item $\prod_1^7\mathfrak{sl}(2)\subset E_7$,
\end{itemize}
there exists a semisimple subalgebra $\mathfrak{h}$ of $\mathfrak{g}$ such that $\mathfrak{h}$ and $\mathfrak{g}$ have same rank and every simple factor of $\mathfrak{h}$ is $\mathfrak{sl}(2)$. Since $G$ is the commutator subgroup of $G_F(F)$, where $G_F$ is a split, simple algebraic group of type $\mathfrak{g}$, there exists a split, semisimple, algebraic subgroup $H_F\subset G_F$ of type $\mathfrak{h}$ (the Zariski closure in $G_F$ of the group generated by $X_\alpha(F)$ for all $\alpha$ belonging to the root subsystem of $\mathfrak{h}$ in the root system of $\mathfrak{g}$) such that $Z(G)\subset Z(H_F)$. Let $m$ be the rank of $G_F$ and $\pi:\widetilde{H}_F\cong\prod_1^m\SL_2\rightarrow H_F$ the universal cover. Since $Z(\widetilde{H}_F)$ surjects on $Z(H_F)$, $Z(\widetilde{H}_F)=\prod_1^mZ(\SL_2)=\prod_1^mZ(\SL_2(F))$, and $-I_2$ is a square, every element of $Z(G)$ is a square in $\pi(\prod_1^m\SL_2(F))=\pi(\prod_1^m[\SL_2(F),\SL_2(F)])\subset [H_F(F),H_F(F)]\subset [G_F(F),G_F(F)]=G$. We are done.\\

Case $\mathfrak{g}=\mathfrak{so}(4n+2)$ ($n\geq2)$:\\
Since $\mathfrak{h}=\mathfrak{so}(6)\times\prod_1^{n-1}\mathfrak{so}(4)$ is a maximal rank semisimple subalgebra of $\mathfrak{g}=\mathfrak{so}(4n+2)$, we find a split, semisimple algebraic subgroup $H_F\subset G_F$ of type $\mathfrak{h}$ such that $Z(G)\subset Z(H_F)$. Let $\pi:\widetilde{H}_F\cong \SL_4\times\prod_1^{n-1}\SL_3\rightarrow H_F$ be the universal cover. Since $Z(G_F(\bar{F}))=\Z/4\Z\subset Z(H_F(\bar{F}))$ and $|Z(\SL_3(\bar{F}))|$ is odd, $\pi$ is injective on $\SL_4$. Since $Z(G)$ is a subgroup of $H_F(F)$ of order a power of $2$, $Z(G)\subset \pi(\SL_4)\cap H_F(F)\cong \SL_4(F)$ by injectivity. Since every element of $\SL_4(F)$ is a product of two squares from above and $\SL_4(F)=[\SL_4(F),\SL_4(F)]\subset [H_F(F),H_F(F)]\subset [G_F(F),G_F(F)]=G$, we are done.
\end{proof}

Let $G$ be a Chevalley group of the form $\SL_2(F)$ for $F$ an infinite field.
If either $m$ or $n$ is congruent to $1$, $2$, or $3$ (mod $4$), then every element of $G$ is of the form $x^m y^n$ since 
$$ \begin{pmatrix} 0&1\\ -1&0\end{pmatrix}^k = \begin{pmatrix} -1&0\\ 0&-1\end{pmatrix}$$ 
if $k\equiv 2$ (mod $4$) (and of course odd powers preserve elements of $Z(G)$).
However, this is not true in general.

\begin{prop}
If $G=\SL_2(F)$ where $F$ is of characteristic zero and  $[F(\zeta_8):F] = 4$ (e.g., $F=\Q$), then $x^4 y^4$ does \emph{not} represent $-I_2$. %$\begin{pmatrix} -1&0\\ 0&-1\end{pmatrix}$
\end{prop}

\begin{proof} If $A^4 = -B^4$ for elements $A,B\in \SL_2(F)$ and $\lambda^{\pm 1}$ and $\mu^{\pm 1}$ are respectively the eigenvalues of $A$ and $B$,
then without loss of generality we may assume $\lambda/\mu = \zeta_8$, and $\Gal(F(\zeta_8)/F)\cong \Z/2\Z\times \Z/2\Z$ acts on $\{\lambda^{\pm 1}\}$
and $\{\mu^{\pm 1}\}$.  One of the automorphisms $\zeta_8\mapsto\zeta_8^3$ or $\zeta_8\mapsto\zeta_8^7$ fixes exactly one of $\lambda$ and $\mu$,
so either $\lambda^2$ or $\mu^2$ lies in $\{\pm  i\}$.  However, $\lambda$ and $\mu$ lie in quadratic extensions of $F$ and by hypothesis, every primitive $8$-th root of unity generates a degree $4$ extension of $F$.
\end{proof}


\begin{thebibliography}{AGKSh}

\bibitem[AGKSh]{AGKSh}
	Avni, Nir; Gelander, Tsachik; Kassabov, Martin; Shalev, Aner:
	Word values in $p$-adic and adelic groups,
	arXiv: 1303.1161v1.
	
\bibitem[Bo1]{Bo1}
	Borel, Armand:
	On free subgroups of semisimple groups,
	\textit{Enseign. Math.} (2) \textbf{29} (1983), no. 1-2, 151–164.
	
%\bibitem[Bo2]{Bo2}
%	Borel, Armand:
%	Linear Algebraic Groups,
%	Graduate Texts in Mathematics, 126 (2nd ed.), \textit{Springer-Verlag} 1991.
	
\bibitem[EG1]{EG1}
	Ellers, E. W.; Gordeev, N. L.:
	Gauss decomposition with prescribed semisimple part in classical Chevalley groups, 
	\textit{Comm. Algebra} \textbf{22} (1994, no.14), 5935-5950. 

\bibitem[EG2]{EG2}
	Ellers, E. W.; Gordeev, N. L.:
	Gauss decomposition with prescribed semisimple part in classical Chevalley groups II: Exceptional cases, 
	\textit{Comm. Algebra} \textbf{23} (1995, no.8), 3085-3098. 
	
\bibitem[EG3]{EG3}
	Ellers, E. W.; Gordeev, N. L.:
	Gauss decomposition with prescribed semisimple part in classical Chevalley groups III: Finite twisted groups, 
	\textit{Comm. Algebra} \textbf{24} (1996, no.14), 4447-4475. 
	
\bibitem[ET]{ET}
	Elkasapy, Abdelrhman; Thom, Andreas:
	About Got\^o's method showing surjectivity of word maps, 
	arXiv: 1207.5596. 

\bibitem[FH]{FH}
	Fulton, William; Harris, Joe:
	Representation Theory, 
	Graduate Texts in Mathematics 129 (1st ed.), Springer-Verlag 1991.
					
\bibitem[Go]{Go}
	Got\^o, Morikuni:
	A theorem on compact semi-simple groups. 
	\textit{J.\ Math.\ Soc.\ Japan} \textbf{1}, (1949), 270--272. 
	
\bibitem[La]{La}
	Larsen, Michael:
	Word maps have large image,
	\textit{Israel\ J.\ Math.} \textbf{139}, (2004), 149--156.	
  	

\bibitem[LaSh1]{LaSh1}
	Larsen, Michael; Shalev, Aner:
	Word maps and Waring type problems,
	\textit{J.\ Amer.\ Math.\ Soc.} \textbf{22} (2009), 437--466.

\bibitem[LaSh2]{LaSh2} Larsen, Michael; Shalev, Aner:
Characters of symmetric groups: sharp bounds and applications,
\textit{Invent.\ Math.} \textbf{174} (2008), no.\ 3, 645--687.

\bibitem[LaSh3]{LaSh3} Larsen, Michael; Shalev, Aner:
On the distribution of values of
certain word maps, {\it Trans. Amer. Math. Soc.}, to appear.


\bibitem[Lev]{Lev} Lev, Arieh:
Products of cyclic conjugacy classes in the groups $\PSL(n,F)$,
\textit{Linear Algebra Appl.} \textbf{179} (1993), 59--83.

\bibitem[LST]{LST} Larsen, Michael; Shalev, Aner; Tiep, Pham:
The Waring problem for finite simple groups, \textit{Annals of Math.} \textbf{174} (2011), 1885--1950.

\bibitem[LST2]{LST2} M. Larsen, A. Shalev and Ph. Tiep, Waring problem for 
finite quasisimple groups, {\it IMRN} rns109 (2012), 26 pages.

	
\bibitem[LOST]{LOST} 
	Liebeck, Martin W.; O'Brien, E. A.; Shalev, Aner; Tiep, Pham Huu: 
The Ore Conjecture, {\it J. Europ. Math. Soc.} {\bf 12} (2010), 
939--1008.	

\bibitem[S]{S} Segal, Dan: {\it Words: notes on verbal width in groups}, 
London Math. Soc. Lecture Note Series {\bf 361}, 
Cambridge University Press, Cambridge, 2009. 


\bibitem[Se]{Se}
	Serre, Jean-Pierre:
	Exemples de plongements des groupes $\PSL_2(\F_p)$ dans des groupes de Lie simples,
	\textit{Invent.\ Math.} \textbf{124}, (1996), 525--562.	 	

\bibitem[Sh1]{Sh1} Shalev, Aner: Word maps, conjugacy classes, and a 
non-commutative Waring-type theorem, {\it Annals of Math.}
{\bf 170} (2009), 1383--1416.


\bibitem[Sh2]{Sh2} Shalev, Aner: Some problems and results in the theory
of word maps, {\it Erd\H{o}s Centennial}, eds: Lov\'{a}sz et al.,
Bolyai Soc. Math. Studies {\bf 25} (2013), 611--649.

\bibitem[St]{St} Steinberg, Robert: Lectures on Chevalley Groups,
Yale University, 1967.

\bibitem[T]{T} Thom, Andreas: Convergent sequences in discrete groups,
\textit{Canad.\ Math.\ Bull.} \textbf{56} (2013), no.\ 2, 424--433. 

\bibitem[Th]{Th} Thompson, Robert C.: Commutators in the special and general linear groups,
\textit{ Trans. Amer. Math. Soc.} \textbf{101} (1961), 16--33. 

\bibitem[VW]{VW} Vaserstein, Leonid N.; Wheland, Ethel: Products of conjugacy classes of two by two matrices,
\textit{ Linear Algebra Appl.} \textbf{230} (1995), 165--188. 
\end{thebibliography}
\end{document}